\documentclass{article}
\usepackage{fullpage}
\usepackage{graphicx}
\usepackage{epstopdf}
\usepackage{latexsym}
\usepackage{amsmath}
\usepackage{amsfonts}
\usepackage{amssymb}
\usepackage{mathrsfs}
\usepackage{pifont}
\usepackage{yhmath}
\usepackage{undertilde}
\usepackage{bbm}
\usepackage{dsfont}
\usepackage{eufrak}
\usepackage{amsthm}
\usepackage[usenames,dvipsnames]{color}
\usepackage{stmaryrd}
\usepackage{wasysym}
\usepackage{bbm}
\usepackage[ruled]{algorithm2e}

\newtheorem{theorem}{Theorem}[section]

\theoremstyle{remark}

\begin{document}
\newcounter{my}
\newenvironment{mylabel}
{
\begin{list}{(\roman{my})}{
\setlength{\parsep}{-1mm}
\setlength{\labelwidth}{8mm}
\usecounter{my}}
}{\end{list}}

\newcounter{my2}
\newenvironment{mylabel2}
{
\begin{list}{(\alph{my2})}{
\setlength{\parsep}{-0mm} \setlength{\labelwidth}{8mm}
\setlength{\leftmargin}{3mm}
\usecounter{my2}}
}{\end{list}}

\newcounter{my3}
\newenvironment{mylabel3}
{
\begin{list}{(\alph{my3})}{
\setlength{\parsep}{-1mm}
\setlength{\labelwidth}{8mm}
\setlength{\leftmargin}{10mm}
\usecounter{my3}}
}{\end{list}}

\vspace{-5em}

\title{\bf A combinatorial property of flows on a cycle}

 \author{Zhuo Diao ${}^{a}$\thanks{Corresponding author. E-mail: diaozhuo@amss.ac.cn.
 Supported by Young Teachers Development Foundation of CUFE (Grant No.QJJ1702)}%\quad Yiyi Jiang ${}^{b}$
}
\date{
%${}^b$ Academy of Mathematics and Systems Science, Chinese Academy of Sciences\\
 %Beijing 100190, China\\
 %${}^c$ School of Mathematical Sciences, University of Chinese Academy of Sciences\\
% Beijing 10049, China\\
 ${}^a$ School of Statistics and Mathematics, Central University of Finance and Economics
Beijing 100081, China\\
 %${}^b$ School of Business, Central University of Finance and Economics\\
%Beijing 100081, China\\
% ${}$\\
%\{xchen,diaozhuo,xdhu\}@amss.ac.cn
%} %\authorrunning{X. Chen, Z. Diao, and X. Hu}
%\date{${}^a$ School of Statistics and Mathematics, Central University of Finance and Economics\\
%${}^b$ Academy of Mathematics and Systems Science, Chinese Academy of Sciences\\
 %Beijing 100190, China\\
 %${}^c$ School of Mathematical Sciences, University of Chinese Academy of Sciences\\
 %Beijing 100049, China\\
 %${}^d$ Department of Computer Science, City University of Hong Kong\\
 %Hong Kong, China
}
\maketitle

\begin{abstract}
In this paper, we prove a combinatorial property of flows on a cycle. $C(V,E)$ is an undirected cycle with two commodities: $\{s_{1},t_{1}\}, \{s_{2},t_{2}\}$;$r_1>0,r_2>0, \mathbf r=(r_i)_{i=1,2}$ and $f,f'$ are both feasible flows for $(C,(s_i,t_i)_{i=1,2},\mathbf r)$. Then $\exists i\in\{1,2\}, p\in P_i, f(p)>0, \forall e\in p, f(e)\geq f'(e)$ ; Here for each $i\in\{1,2\}$, let $P_i$ be the set of $s_i$-$t_i$ paths in $C$ and $P=\cup_{i=1,2}P_i$. This means given a two-commodity instance on a cycle, any two distinct network flow $f$ and $f'$, compared with $f'$, $f$ can't decrease every path's flow amount at the same time. This combinatorial property is a generalization from single-commodity case to two-commodity case, and we also give an instance to illustrate the combinatorial property doesn't hold on for $k-$commodity case when $k\geq 3$.
\end{abstract}


\section{Introduction}

Network flow is a very hot topic in graph theory and combinatorial optimization. There are a lot of literatures about network flow theory\cite{beckmann1956}\cite{braess1968}\cite{hl1997}\cite{hl2003}\cite{hm2014}\cite{lrtw2011}\cite{milchtaich2006}\cite{rt2002}.\\\\
Formally, $G(V,E)$ is an undirected graph. $\{s_{i},t_{i}\},i\in \{1,...,k\}$ are $k$ origin-terminal pairs, here $k$ is a positive integer, and for any $i\in \{1,...,k\}, s_{i}\neq t_{i}$, there is at least one path from $s_{i}$ to $t_{i}$ in $G(V,E)$. For any $i\in \{1,...,k\}$, $P_{i}$ is the set of paths from $s_{i}$ to $t_{i}$ in $G(V,E)$ and $P =\cup_{i\in \{1,...,k\}}P_{i}$. For any ~$i\in \{1,...,k\}$, we need to transmit a traffic of $r_{i}>0$ from $s_{i}$ to $t_{i}$ in $G(V,E)$. $<G(V,E),\{s_{i},t_{i}\},r_{i}>0,i\in \{1,...,k\}>$ is a flow instance. $f: P\rightarrow \mathbb{R}_{+}\geq 0$ satisfying for any $i\in \{1,...,k\}$, $\sum_{p\in P_{i}}f(p)= r_{i}$ is called a feasible flow. Correspondingly, for any edge $e\in E, f(e)=\sum_{p\in P_{i},e\in p}f(p)$ is the flow amount on the edge $e$ under the feasible flow $f$.\\\\
Our contribution: In this paper, we consider the flow allocation on a cycle.\\\\
First, it is easy to see a combinatorial property: given a single-commodity instance on a cycle, any two distinct network flows $f$ and $f'$, compared with $f'$, $f$ can't decrease every path's flow amount at the same time.\\\\
Second, the combinatorial property is generalized from single-commodity case to two-commodity case, which says given a two-commodity instance on a cycle, any two distinct network flow $f$ and $f'$, compared with $f'$, $f$ can't decrease every path's flow amount at the same time.\\\\
%\begin{theorem}\label{flow allocation}{ $C(V,E)$ is an undirected cycle with two commodities: $\{s_{1},t_{1}\}, \{s_{2},t_{2}\}$;
%$r_1>0,r_2>0, \mathbf r=(r_i)_{i=1,2}$ and $f,f'$ are both feasible flows for $(C,(s_i,t_i)_{i=1,2},\mathbf r)$. Then $\exists i\in\{1,2\}, p\in P_i, %f(p)>0, \forall e\in p, f(e)\geq f'(e)$ ; Here for each $i\in\{1,2\}$, let $P_i$ be the set of $s_i$-$t_i$ paths in $C$ and $\mathcal %P=\cup_{i=1,2}P_i$.}
%\end{theorem}
Last, we give an instance to illustrate the combinatorial property doesn't hold on for $k-$commodity case when $k\geq 3$.
%The above theorem means given a two-commodity instance on a cycle, any two distinct network flow $f$ and $f'$, compared with $f'$, $f$ can't decrease every path's flow amount at the same time.
\section{A combinatorial property of flows on a cycle}

Here we discuss a combinatorial property of flows on a cycle.

\subsection{One-commodity flow on a cycle}

First, it is easy to see a combinatorial property: given a single-commodity instance on a cycle, any two distinct network flows $f$ and $f'$, compared with $f'$, $f$ can't decrease every path's flow amount at the same time. We formally describe this result as follows:

\begin{theorem}\label{flow allocation1}{ $C(V,E)$ is an undirected cycle with single-commodity: $\{s,t\};r>0$ and $f,f'$ are both feasible flows for $(C,(s,t),r)$. Then $\exists p\in P, f(p)>0, \forall e\in p, f(e)\geq f'(e)$ ; Here let $P$ be the set of $s$-$t$ paths in $C$.}
\end{theorem}

\begin{proof} Now we denote these two edge-disjoint $s-t$ paths in $P$ as $p_{1},p_{2}$. $f$ and $f'$ are both feasible flows, thus $f(p_{1})+ f(p_{2})= f'(p_{1})+ f'(p_{2})= r$. Thus $f(p_{1})\geq f'(p_{1})$ or $f(p_{2})\geq f'(p_{2})$. Without loss of generality, we can assume $f(p_{1})\geq f'(p_{1})$, then $\forall e\in p_{1}, f(e)=f(p_{1})\geq f'(p_{1})=f'(e)$. If $f(p_{1})>0$, our theorem holds on. If $f(p_{1})=0$, then $f(p_{2})= r\geq f'(p_{2}), f(p_{2})>0, \forall e\in p_{2}, f(e)=f(p_{2})\geq f'(p_{2})=f'(e)$, our theorem also holds on.
\end{proof}

\subsection{Two-commodity flow on a cycle}

Second, the combinatorial property is generalized from single-commodity case to two-commodity case, which says given a two-commodity instance on a cycle, any two distinct network flow $f$ and $f'$, compared with $f'$, $f$ can't decrease every path's flow amount at the same time. We formally describe this result as follows:

\begin{theorem}\label{flow allocation2}{ $C(V,E)$ is an undirected cycle with two commodities: $\{s_{1},t_{1}\}, \{s_{2},t_{2}\}$;
$r_1>0,r_2>0, \mathbf r=(r_i)_{i=1,2}$ and $f,f'$ are both feasible flows for $(C,(s_i,t_i)_{i=1,2},\mathbf r)$. Then $\exists i\in\{1,2\}, p\in P_i, f(p)>0, \forall e\in p, f(e)\geq f'(e)$ ; Here for each $i\in\{1,2\}$, let $P_i$ be the set of $s_i$-$t_i$ paths in $C$ and $P=\cup_{i=1,2}P_i$.}
\end{theorem}

%\noindent$\displaystyle Explanation:$ If the theorem above holds on, then for every instance $(C,(s_i,t_i)_{i=1,2},\mathbf r,\ell)$, let us take $f'=f_{NE}$. Then for every feasible flow $f$ for $(C,(s_i,t_i)_{i=1,2},\mathbf r,\ell)$, we have $\exists i\in\{1,2\}, p\in P_i, f(p)>0, \forall e\in p, f(e)\geq f_{NE}(e)$. Thus $C_{p}(f)\geq C_{p}(f_{NE})\geq \ell_i(C,\mathbf r)$. So $f_{NE}$ is weakly pareto optimal. Thus Lemma is proven.\\

\begin{proof} Let us denote $V^{*}=\cup_{i=1,2}\{s_{i},t_{i}\}$\\

\noindent$\displaystyle Case$1$:$ $|V^{*}|=2$, then $\{s_{1},t_{1}\}=\{s_{2},t_{2}\}$, according to the symmetry of terminals and flows, we just need to prove when $s_{1}=s_{2},t_{1}=t_{2}$, our theorem holds on. Now we denote these two edge-disjoint $s_{1}-t_{1}(s_{2}-t_{2})$ paths in $P_1(P_2)$ as $p_{1},p_{2}$. $f$ and $f'$ are both feasible flows, thus $f(p_{1})+ f(p_{2})= f'(p_{1})+ f'(p_{1})= r_{1}+r_{2}$. Thus $f(p_{1})\geq f'(p_{1})$ or $f(p_{2})\geq f'(p_{2})$. Without loss of generality, we can assume $f(p_{1})\geq f'(p_{1})$, then $\forall e\in p_{1}, f(e)=f(p_{1})\geq f'(p_{1})=f'(e)$. If $f(p_{1})>0$, our theorem holds on. If $f(p_{1})=0$, then $f(p_{2})= r_{1}+r_{2}\geq f'(p_{2}), f(p_{2})>0, \forall e\in p_{2}, f(e)=f(p_{2})\geq f'(p_{2})=f'(e)$, our theorem also holds on.\\

\noindent$\displaystyle Case $2$:$ $|V^{*}|=3$, according to the symmetry of terminals and flows, we just need to prove when $s_{1}=s_{2}\neq t_{1}\neq t_{2}$, our theorem holds on. Here we denote $P_1=\{ l_{1}, l_{2}\cup l_{3}\}, P_2=\{ l_{2}, l_{1}\cup l_{3}\}$, which is shown in Figure~\ref{lemma2}(a). $f$ and $f'$ are both feasible flows for $(G,(s_i,t_i)_{i=1,2},\mathbf r)$. Thus we have next equations:

\begin{figure}[htbp]
\begin{center}
\includegraphics[scale=0.3]{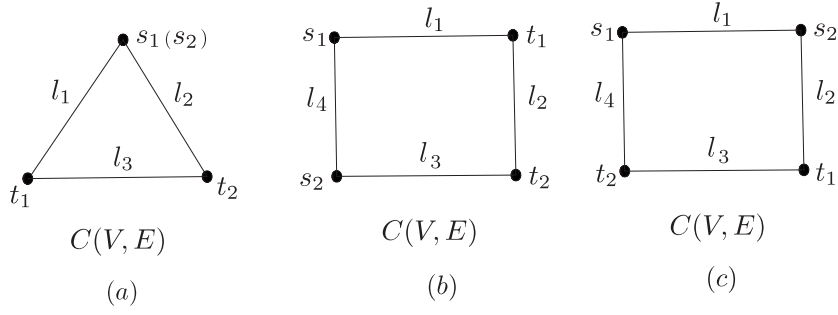}
\caption{\label{lemma2}$2$-commodity $(C,(s_i,t_i)_{i=1}^2)$ }
\end{center}
\end{figure}

\begin{equation} f(l_{1})+ f(l_{2}\cup l_{3})= f'(l_{1})+ f'(l_{2}\cup l_{3})= r_{1};
\end{equation}

\begin{equation} f(l_{2})+ f(l_{1}\cup l_{3})= f'(l_{2})+ f'(l_{1}\cup l_{3})= r_{2};
\end{equation}

\begin{equation} \begin{aligned}   \forall e\in l_{1}, f(e)= f(l_{1})+ f(l_{1}\cup l_{3})= f(l_{1})+ r_{2}- f(l_{2});\\
                                                          f'(e)= f'(l_{1})+ f'(l_{1}\cup l_{3})= f'(l_{1})+ r_{2}- f'(l_{2});
\end{aligned}\end{equation}

\begin{equation} \begin{aligned}   \forall e\in l_{2}, f(e)= f(l_{2})+ f(l_{2}\cup l_{3})= f(l_{2})+ r_{1}- f(l_{1});\\
                                                          f'(e)= f'(l_{2})+ f'(l_{2}\cup l_{3})= f'(l_{2})+ r_{1}- f'(l_{1});
\end{aligned}\end{equation}

\begin{equation} \begin{aligned}   \forall e\in l_{3}, f(e)= f(l_{1}\cup l_{3})+ f(l_{2}\cup l_{3})= r_{1}+ r_{2}- f(l_{1})- f(l_{2});\\
                                                          f'(e)= f'(l_{1}\cup l_{3})+ f'(l_{2}\cup l_{3})= r_{1}+ r_{2}- f'(l_{1})- f'(l_{2});
\end{aligned}\end{equation}

\noindent$\displaystyle When$ $f(l_{1})>0$: \\

If $f(l_{1})- f(l_{2})\geq f'(l_{1})- f'(l_{2})$, then $\forall e\in l_{1}, f(e)\geq f'(e)$, our theorem holds on; \\

If $f(l_{1})- f(l_{2})< f'(l_{1})- f'(l_{2})$, then $\forall e\in l_{2}, f(e)> f'(e)$. Now if $f(l_{2})>0$, our theorem holds on;\\

If $f(l_{1})- f(l_{2})< f'(l_{1})- f'(l_{2})$ and $f(l_{2})=0$, we have $f(l_{1})+ f(l_{2})= f(l_{1})- f(l_{2})< f'(l_{1})- f'(l_{2})\leq f'(l_{1})+ f'(l_{2})$,
then $\forall e\in l_{2}, f(e)> f'(e), \forall e\in l_{3}, f(e)> f'(e)$. Now if $f(l_{2}\cup l_{3})>0$, our theorem holds on;\\

If $f(l_{1})- f(l_{2})< f'(l_{1})- f'(l_{2})$, $f(l_{2})=0$ and $f(l_{2}\cup l_{3})=0$, then $f(l_{1})=r_{1}, f(l_{1})- f(l_{2})= r_{1}\geq f'(l_{1})- f'(l_{2})$, this is a contradiction with our assumptions.\\

\noindent$\displaystyle When$ $f(l_{2})>0$, this is totally a symmetric case with $f(l_{1})>0$, we do not repeat it here.\\

\noindent$\displaystyle When$ $f(l_{1})= f(l_{2})=0$, then $f(l_{2}\cup l_{3})=r_{1}, f(l_{1}\cup l_{3})=r_{2}$ and $\forall e\in l_{3}, f(e)= r_{1}+r_{2}\geq f'(e)$. Because either $f(l_{1})- f(l_{2})\geq f'(l_{1})- f'(l_{2})$ or $f(l_{1})- f(l_{2})\leq f'(l_{1})- f'(l_{2})$, we have either $\forall e\in l_{1}, f(e)\geq f'(e)$ or $\forall e\in l_{2}, f(e)\geq f'(e)$, thus our theorem holds on.\\

\noindent$\displaystyle Case $3$:$ $|V^{*}|=4$, according to the symmetry of terminals and flows, we just need to prove in Figure~\ref{lemma2}(b) and (c), our theorem holds on.\\

\noindent$\displaystyle In$ Figure~\ref{lemma2}(b): we denote $P_1=\{ l_{1}, l_{2}\cup l_{3}\cup l_{4}\}, P_2=\{ l_{3}, l_{1}\cup l_{2}\cup l_{4}\}$, which is shown in Figure~\ref{lemma2}(b). $f$ and $f'$ are both feasible flows for $(G,(s_i,t_i)_{i=1,2},\mathbf r)$. Thus we have next equations:

\begin{equation} f(l_{1})+ f(l_{2}\cup l_{3}\cup l_{4})= f'(l_{1})+ f'(l_{2}\cup l_{3}\cup l_{4})= r_{1};
\end{equation}

\begin{equation} f(l_{3})+ f(l_{1}\cup l_{2}\cup l_{4})= f'(l_{3})+ f'(l_{1}\cup l_{2}\cup l_{4})= r_{2};
\end{equation}

\begin{equation} \begin{aligned}   \forall e\in l_{1}, f(e)= f(l_{1})+ f(l_{1}\cup l_{2}\cup l_{4})= f(l_{1})+ r_{2}- f(l_{3});\\
                                                          f'(e)= f'(l_{1})+ f'(l_{1}\cup l_{2}\cup l_{4})= f'(l_{1})+ r_{2}- f'(l_{3});
\end{aligned}\end{equation}

\begin{equation} \begin{aligned}   \forall e\in l_{2}, f(e)= f(l_{1}\cup l_{2}\cup l_{4})+ f(l_{2}\cup l_{3}\cup l_{4})= r_{1}+ r_{2}- f(l_{1})- f(l_{3});\\
                                                        f'(e)= f'(l_{1}\cup l_{2}\cup l_{4})+ f'(l_{2}\cup l_{3}\cup l_{4})= r_{1}+ r_{2}- f'(l_{1})- f'(l_{3});
\end{aligned}\end{equation}

\begin{equation} \begin{aligned}   \forall e\in l_{3}, f(e)= f(l_{2}\cup l_{3}\cup l_{4})+ f(l_{3})= r_{1}- f(l_{1})+ f(l_{3});\\
                                                         f'(e)= f'(l_{2}\cup l_{3}\cup l_{4})+ f'(l_{3})= r_{1}- f'(l_{1})+ f'(l_{3});
\end{aligned}\end{equation}

\begin{equation} \begin{aligned}   \forall e\in l_{4}, f(e)= f(l_{1}\cup l_{2}\cup l_{4})+ f(l_{2}\cup l_{3}\cup l_{4})= r_{1}+ r_{2}- f(l_{1})- f(l_{3});\\
                                                        f'(e)= f'(l_{1}\cup l_{2}\cup l_{4})+ f'(l_{2}\cup l_{3}\cup l_{4})= r_{1}+ r_{2}- f'(l_{1})- f'(l_{3});
\end{aligned}\end{equation}

\noindent$\displaystyle When$ $f(l_{1})>0$: \\

If $f(l_{1})- f(l_{3})\geq f'(l_{1})- f'(l_{3})$, then $\forall e\in l_{1}, f(e)\geq f'(e)$, our theorem holds on; \\

If $f(l_{1})- f(l_{3})< f'(l_{1})- f'(l_{3})$, then $\forall e\in l_{3}, f(e)> f'(e)$. Now if $f(l_{3})>0$, our theorem holds on;\\

If $f(l_{1})- f(l_{3})< f'(l_{1})- f'(l_{3})$ and $f(l_{3})=0$, we have $f(l_{1})+ f(l_{3})= f(l_{1})- f(l_{3})< f'(l_{1})- f'(l_{3})\leq f'(l_{1})+ f'(l_{3})$,
then $\forall e\in l_{2}\cup l_{4}, f(e)> f'(e), \forall e\in l_{3}, f(e)> f'(e)$. Now if $f(l_{2}\cup l_{3}\cup l_{4})>0$, our theorem holds on;\\

If $f(l_{1})- f(l_{3})< f'(l_{1})- f'(l_{3})$, $f(l_{3})=0$ and $f(l_{2}\cup l_{3}\cup l_{4})=0$, then $f(l_{1})=r_{1}, f(l_{1})- f(l_{3})= r_{1}\geq f'(l_{1})- f'(l_{3})$, this is a contradiction with our assumptions.\\

\noindent$\displaystyle When$ $f(l_{3})>0$, this is totally a symmetric case with $f(l_{1})>0$, we do not repeat it here.\\

\noindent$\displaystyle When$ $f(l_{1})= f(l_{3})=0$, then $f(l_{2}\cup l_{3}\cup l_{4})=r_{1}, f(l_{1}\cup l_{2}\cup l_{4})=r_{2}$ and $\forall e\in l_{2}\cup l_{4}, f(e)= r_{1}+r_{2}\geq f'(e)$. Because either $f(l_{1})- f(l_{3})\geq f'(l_{1})- f'(l_{3})$ or $f(l_{1})- f(l_{3})\leq f'(l_{1})- f'(l_{3})$, we have either $\forall e\in l_{1}, f(e)\geq f'(e)$ or $\forall e\in l_{3}, f(e)\geq f'(e)$, thus our theorem holds on.\\

\noindent$\displaystyle In$ Figure~\ref{lemma2}(c): we denote $P_1=\{ l_{1}\cup l_{2}, l_{3}\cup l_{4}\}, P_2=\{ l_{1}\cup l_{4}, l_{2}\cup l_{3}\}$, which is shown in Figure~\ref{lemma2}(c). $f$ and $f'$ are both feasible flows for $(G,(s_i,t_i)_{i=1,2},\mathbf r)$. Thus we have next equations:

\begin{equation} f(l_{1}\cup l_{2})+ f(l_{3}\cup l_{4})= f'(l_{1}\cup l_{2})+ f'(l_{3}\cup l_{4})= r_{1};
\end{equation}

\begin{equation} f(l_{1}\cup l_{4})+ f(l_{1}\cup l_{2})= f'(l_{1}\cup l_{4})+ f'(l_{2}\cup l_{3})= r_{2};
\end{equation}

\begin{equation} \begin{aligned}   \forall e\in l_{1}, f(e)= f(l_{1}\cup l_{2})+ f(l_{1}\cup l_{4})= f(l_{1}\cup l_{2})+ r_{2}- f(l_{2}\cup l_{3});\\
                                                          f'(e)= f'(l_{1}\cup l_{2})+ f'(l_{1}\cup l_{4})= f'(l_{1}\cup l_{2})+ r_{2}- f'(l_{2}\cup l_{3});
\end{aligned}\end{equation}

\begin{equation} \begin{aligned}   \forall e\in l_{2}, f(e)= f(l_{1}\cup l_{2})+ f(l_{2}\cup l_{3});\\
                                                        f'(e)= f'(l_{1}\cup l_{2})+ f'(l_{2}\cup l_{3});
\end{aligned}\end{equation}

\begin{equation} \begin{aligned}   \forall e\in l_{3}, f(e)= f(l_{2}\cup l_{3})+ f(l_{3}\cup l_{4})= f(l_{2}\cup l_{3})+ r_{1}- f(l_{1}\cup l_{2});\\
                                                          f'(e)= f'(l_{2}\cup l_{3})+ f'(l_{3}\cup l_{4})= f'(l_{2}\cup l_{3})+ r_{1}- f'(l_{1}\cup l_{2});
\end{aligned}\end{equation}

\begin{equation} \begin{aligned}   \forall e\in l_{4}, f(e)= f(l_{1}\cup l_{4})+ f(l_{3}\cup l_{4})= r_{1}+ r_{2}- f(l_{1}\cup l_{2})- f(l_{2}\cup l_{3});\\
                                                       f'(e)= f'(l_{1}\cup l_{4})+ f'(l_{3}\cup l_{4})= r_{1}+ r_{2}- f'(l_{1}\cup l_{2})- f'(l_{2}\cup l_{3});
\end{aligned}\end{equation}

According to the symmetry of flows, we just need to prove when $f(l_{1}\cup l_{2})>0$, our theorem holds on.\\

\noindent$\displaystyle When$ $f(l_{1}\cup l_{2})>0$:\\

If $\forall e\in l_{1}, f(e)\geq f'(e)$ and $\forall e\in l_{2}, f(e)\geq f'(e)$, our theorem holds on.\\

If $\forall e\in l_{1}, f(e)< f'(e)$ and $f(l_{2}\cup l_{3})>0$, then $f(l_{1}\cup l_{2})- f(l_{2}\cup l_{3})< f'(l_{1}\cup l_{2})- f'(l_{2}\cup l_{3}); \forall e\in l_{3}, f(e)> f'(e)$. Now if $\forall e\in l_{2}, f(e)\geq f'(e)$, our theorem holds on.\\

If $\forall e\in l_{1}, f(e)< f'(e)$,$f(l_{2}\cup l_{3})>0$ and $\forall e\in l_{2}, f(e)< f'(e)$, then $f(l_{1}\cup l_{2})- f(l_{2}\cup l_{3})< f'(l_{1}\cup l_{2})- f'(l_{2}\cup l_{3}); \forall e\in l_{3}, f(e)> f'(e)$;$f(l_{1}\cup l_{2})+ f(l_{2}\cup l_{3})< f'(l_{1}\cup l_{2})+ f'(l_{2}\cup l_{3}); \forall e\in l_{4}, f(e)> f'(e)$. Now if $f(l_{3}\cup l_{4})>0$, our theorem holds on.\\

If $\forall e\in l_{1}, f(e)< f'(e)$,$f(l_{2}\cup l_{3})>0$,$\forall e\in l_{2}, f(e)< f'(e)$ and $f(l_{3}\cup l_{4})=0$, then $f(l_{1}\cup l_{2})= r_{1}, f(l_{1}\cup l_{2})- f(l_{2}\cup l_{3})< f'(l_{1}\cup l_{2})- f'(l_{2}\cup l_{3})$. Thus $f(l_{2}\cup l_{3})> f'(l_{2}\cup l_{3})$. Thus $f(l_{1}\cup l_{2})+ f(l_{2}\cup l_{3})> f'(l_{1}\cup l_{2})+ f'(l_{2}\cup l_{3})$, which is a contradiction with $\forall e\in l_{2}, f(e)< f'(e)$.\\

If $\forall e\in l_{1}, f(e)< f'(e)$ and $f(l_{2}\cup l_{3})=0$, then $\forall e\in l_{3}, f(e)> f'(e), f(l_{1}\cup l_{2})= f(l_{1}\cup l_{2})- f(l_{2}\cup l_{3})< f'(l_{1}\cup l_{2})- f'(l_{2}\cup l_{3})\leq f'(l_{1}\cup l_{2})\leq r_{1}$, thus we have $f(l_{3}\cup l_{4})= r_{1}-f(l_{1}\cup l_{2})>0$. Now if $\forall e\in l_{4}, f(e)\geq f'(e)$, our theorem holds on.\\

If $\forall e\in l_{1}, f(e)< f'(e)$,$f(l_{2}\cup l_{3})=0$ and $\forall e\in l_{4}, f(e)< f'(e)$, then $f(l_{1}\cup l_{2})- f(l_{2}\cup l_{3})< f'(l_{1}\cup l_{2})- f'(l_{2}\cup l_{3})$,$f(l_{1}\cup l_{2})+ f(l_{2}\cup l_{3})> f'(l_{1}\cup l_{2})+ f'(l_{2}\cup l_{3})$ and $f(l_{2}\cup l_{3})=0$ ,this is impossible.\\

If $\forall e\in l_{1}, f(e)\geq f'(e)$ and $\forall e\in l_{2}, f(e)< f'(e)$, then $f(l_{1}\cup l_{2})- f(l_{2}\cup l_{3})\geq f'(l_{1}\cup l_{2})- f'(l_{2}\cup l_{3})$,$f(l_{1}\cup l_{2})+ f(l_{2}\cup l_{3})< f'(l_{1}\cup l_{2})+ f'(l_{2}\cup l_{3})$. Thus $\forall e\in l_{3}, f(e)\leq f'(e)$ and $\forall e\in l_{4},
f(e)> f'(e)$. Now if $f(l_{1}\cup l_{4})>0$, our theorem holds on.\\

If $\forall e\in l_{1}, f(e)\geq f'(e)$,$\forall e\in l_{2}, f(e)< f'(e)$ and $f(l_{1}\cup l_{4})=0$, then we have $f(l_{1}\cup l_{2})- f(l_{2}\cup l_{3})\geq f'(l_{1}\cup l_{2})- f'(l_{2}\cup l_{3})$,$f(l_{1}\cup l_{2})+ f(l_{2}\cup l_{3})< f'(l_{1}\cup l_{2})+ f'(l_{2}\cup l_{3})$ and $f(l_{2}\cup l_{3})=r_{2}\geq f'(l_{2}\cup l_{3})$. Thus $f(l_{1}\cup l_{2})\geq f'(l_{1}\cup l_{2}),f(l_{1}\cup l_{2})< f'(l_{1}\cup l_{2})$ which is impossible.\\

Above all, whatever cases our theorem holds on, so we finish the proof of Theorem.

\end{proof}

\subsection{$k$-commodity flow on a cycle with $k\geq 3$}

The above theorem is a little surprising. It means given a two-commodity instance on a cycle, any two distinct network flow $f$ and $f'$, compared with $f'$, $f$ can't decrease every path's flow amount at the same time. A similar result can not hold for $k$-commodity case when $k\geq 3$ and we give out an instance as following:\\\\
$C=(V,E)$ is a cycle of length $6$, as shown in figure~\ref{hexagon}. Consider three-commodity instance $(C,(s_i,t_i)_{i=1}^3,\mathbf r)$ with $\mathbf r=(3,3,3)$,

\begin{figure}[htbp]
\begin{center}
\includegraphics[scale=0.25]{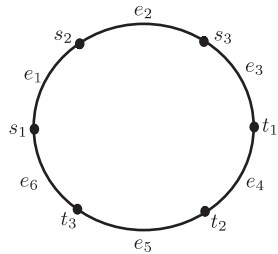}
\caption{\label{hexagon}$3$-commodity instance $(C,(s_i,t_i)_{i=1}^3)$.}
\end{center}
\end{figure}

Notice that $P_1=\{ e_{1}\cup e_{2}\cup e_{3},  e_{6}\cup e_{5}\cup e_{4}\}$,  $P_2= \{ e_{2}\cup e_{3}\cup e_{4},  e_{1}\cup e_{6}\cup e_{5}\}$ and $P_3=\{   e_{3}\cup e_{4}\cup e_{5},e_{2}\cup e_{1}\cup e_{6}\}$. It is easy to verify there exist two feasible flows $f$ and $f'$ in $(C,(s_i,t_i)_{i=1,2,3},\mathbf r)$：\\
$f(e_{1}\cup e_{2}\cup e_{3})=2$,  $f(e_{6}\cup e_{5}\cup e_{4})=1$, $f(e_{2}\cup e_{3}\cup e_{4})=1$,  $f(e_{1}\cup e_{6}\cup e_{5})=2$, $f(e_{3}\cup e_{4}\cup e_{5})=2$,  $f(e_{2}\cup e_{1}\cup e_{6})=1$;
$f'(e_{1}\cup e_{2}\cup e_{3})=1$,  $ f'(e_{6}\cup e_{5}\cup e_{4})=2$, $f'(e_{2}\cup e_{3}\cup e_{4})=2$,  $ f'(e_{1}\cup e_{6}\cup e_{5})=1$, $f'(e_{3}\cup e_{4}\cup e_{5})=1$,  $ f'(e_{2}\cup e_{1}\cup e_{6})=2$. The next table gives out each edge's flows under $f$ and $f'$:   %下面表格给出了上述两种可行流下各边的流量值和费用函数值：
\begin{center}
 \begin{tabular}{|c|c|c|c|c|c|c|}
  \hline
   % after \\: \hline or \cline{col1-col2} \cline{col3-col4} ...
   $j$ & 1 & 2 & 3 & 4 & 5 & 6\\
   \hline
   $f(e_j)$ & 5 & 4 & 5 & 4 & 5 & 4\\
   %$\ell_{e_j}(\pi(e_j))$ & 1 & 1 & 1 & 1 & 1 & 1\\
   $f'(e_j)$ & 4 & 5 & 4 & 5 & 4 & 5\\
  %$\ell_{e_j}(f(e_j))$ & 0 & 1 & 0 & 1 & 0 & 1\\
  \hline
 \end{tabular}
\end{center}

From this table, it is easy to see for any path $p$ in $P$, the flow $f$ has some edges with less flow amount than the flow $f'$. Thus the result in theorem~\ref{flow allocation2} doesn't hold on for $k$-commodity case when $k\geq 3$.

\section{Conclusion}

Conclusion: In this paper, we prove a combinatorial property of flows on a cycle. $C(V,E)$ is an undirected cycle with two commodities: $\{s_{1},t_{1}\}, \{s_{2},t_{2}\}$;$r_1>0,r_2>0, \mathbf r=(r_i)_{i=1,2}$ and $f,f'$ are both feasible flows for $(C,(s_i,t_i)_{i=1,2},\mathbf r)$. Then $\exists i\in\{1,2\}, p\in P_i, f(p)>0, \forall e\in p, f(e)\geq f'(e)$ ; Here for each $i\in\{1,2\}$, let $P_i$ be the set of $s_i$-$t_i$ paths in $C$ and $P=\cup_{i=1,2}P_i$. This means given a two-commodity instance on a cycle, any two distinct network flow $f$ and $f'$, compared with $f'$, $f$ can't decrease every path's flow amount at the same time. This combinatorial property is a generalization from single-commodity case to two-commodity case, and we also give an instance to illustrate the combinatorial property doesn't hold on for $k-$commodity case when $k\geq 3$.
%Future work: $(1)$ we can consider the combinatorial property of flows on other graphic structures.\\\\
%$(2)$ The combinatorial property of flows is related with the conception of weak pareto optimality.
\bibliography{ref}

\end{document}